\documentclass{amsart}

\usepackage{amsfonts,amssymb,verbatim,amsmath,amsthm,latexsym,textcomp,amscd, hyperref}
\usepackage{epsfig}
\usepackage{amsmath,amsthm,graphics,float}
\usepackage{graphicx}
\usepackage{xcolor}
\usepackage{url}
\usepackage{thm-restate}
\usepackage{amsfonts,amssymb,verbatim,amsmath,amsthm,latexsym,textcomp,amscd}
\usepackage{tikz,tikz-cd}
\usepackage{graphics,textcomp,hyperref}
\usepackage{graphicx}
\usepackage{url}
\usepackage{psfrag}
\usepackage{stackrel}
\usepackage{subfiles}

\begin{document}
	
	\newtheorem{theorem}{Theorem}[section]
	\newtheorem{prop}[theorem]{Proposition}
	\newtheorem{lemma}[theorem]{Lemma}
	\newtheorem{cor}[theorem]{Corollary}
	\newtheorem{cond}[theorem]{Condition}
	\newtheorem{ing}[theorem]{Ingredients}
	\newtheorem{conj}[theorem]{Conjecture}
	\newtheorem{claim}[theorem]{Claim}
	\newtheorem{constr}[theorem]{Construction}
	\newtheorem{rem}[theorem]{Remark}
	\newtheorem{defn}[theorem]{Definition}
	\newtheorem{eg}[theorem]{Example}
	\newtheorem{rmk}[theorem]{Remark}
	
	\newtheorem*{theorem*}{Theorem}
	\newtheorem*{modf}{Modification for arbitrary $n$}
	\newtheorem{qn}[theorem]{Question}
	\newtheorem{condn}[theorem]{Condition}
	\newtheorem*{BGIT}{Bounded Geodesic Image Theorem}
	\newtheorem*{BI}{Behrstock Inequality}
	\newtheorem*{QCH}{Wise's Quasiconvex Hierarchy Theorem}

	\numberwithin{equation}{section}
	\newtheorem{lem}[equation]{Lemma}
	\newtheorem{sublem}[equation]{Sublemma}

	\newtheorem{que}[equation]{Question}
	\newtheorem{ob}[equation]{Observation}
	\newtheorem{definition}[equation]{Definition}
	
	\newtheorem{remark}[equation]{Remark}
	\newtheorem{example}[equation]{Example}

	\newcommand*{\inc}{\ensuremath{\mathcal{I}}}
	\newcommand*{\im}{\ensuremath{\textmd{Im}\ }}
	
	\newcommand{\A}{{\mathcal A}}
	\newcommand{\B}{\mathcal{B}}
	\newcommand{\C}{{\mathcal C}}
	\renewcommand{\c}{\mathbb C}
	\newcommand{\D}{\mathbb{D}}
	\renewcommand{\d}{\mathcal{D}}
	\newcommand{\e}{{\mathcal E}}
	\newcommand{\E}{\mathbb E}
	\newcommand{\F}{{\mathcal F}}
	\newcommand{\G}{{\mathcal G}}
	\renewcommand{\H}{\mathbb H}
	\newcommand{\h}{{\mathcal H}}
	\newcommand{\I}{{\mathcal I}}
	\newcommand{\K}{{\mathcal K}}
	\renewcommand{\L}{{\mathcal L}}
	\newcommand{\M}{{\mathcal M}}
	\newcommand{\N}{\mathbb N}
	\newcommand{\n}{\mathcal{N}}
	\newcommand{\p}{{\mathcal P}}
	\renewcommand{\P}{{\mathcal P}}
	\newcommand{\Q}{\mathbb Q}
	\newcommand{\R}{\mathbb R}
	\newcommand{\calr}{{\mathcal R}}
	\renewcommand{\S}{{\mathcal S}}
	\newcommand{\T}{{\mathcal T}}
	\newcommand{\U}{{\mathcal U}}
	\newcommand{\V}{{\mathcal V}}
	\newcommand{\W}{{\mathcal W}}
	\newcommand{\w}{{\mathcal W}}
	\newcommand{\Z}{\mathbb Z}
	\newcommand{\geodim}{\operatorname{GeomDim}}
	\newcommand{\topdim}{\operatorname{TopDim}}
	
	\newcommand{\acts}{\curvearrowright}
	\newcommand{\al}{\alpha}
	\newcommand{\be}{\beta}
	\newcommand{\cat}{\operatorname{CAT}}
	\newcommand{\cc}{\operatorname{CC}}
	\newcommand{\ccc}{\operatorname{CCC}}
	\newcommand{\De}{\Delta}
	\newcommand{\diam}{\operatorname{diam}}
	\newcommand{\dist}{\operatorname{dist}}
	\newcommand{\face}{\operatorname{Face}}
	\newcommand{\ga}{\gamma}
	\newcommand{\Ga}{\Gamma}
	\newcommand{\id}{\operatorname{id}}

	\newcommand{\map}{\rightarrow}
	\newcommand{\boundary}{\partial}
	\newcommand{\integers}{{\mathbb Z}}
	\newcommand{\natls}{{\mathbb N}}
	\newcommand{\ratls}{{\mathbb Q}}
	\newcommand{\reals}{{\mathbb R}}
	\newcommand{\proj}{{\mathbb P}}
	\newcommand{\lhp}{{\mathbb L}}
	\newcommand{\tr}{{\operatorname{Tread}}}
	\newcommand{\rs}{{\operatorname{Riser}}}
	\newcommand{\tube}{{\mathbb T}}
	\newcommand{\cusp}{{\mathbb P}}
	\newcommand\AAA{{\mathcal A}}
	\newcommand\BB{{\mathcal B}}
	\newcommand\CC{{\mathcal C}}
	\newcommand\ccd{{{\mathcal C}_\Delta}}
	\newcommand\DD{{\mathcal D}}
	\newcommand\EE{{\mathcal E}}
	\newcommand\FF{{\mathcal F}}
	
	\newcommand\GG{{\mathcal G}}
	\newcommand\GGs{{{\mathcal G}^*}}
	\newcommand\HH{{\mathcal H}}
	\newcommand\II{{\mathcal I}}
	\newcommand\JJ{{\mathcal J}}
	\newcommand\KK{{\mathcal K}}
	\newcommand\LL{{\mathcal L}}
	\newcommand\MM{{\mathcal M}}
	\newcommand\NN{{\mathcal N}}
	\newcommand\OO{{\mathcal O}}
	\newcommand\PP{{\mathcal P}}
	\newcommand\QQ{{\mathcal Q}}
	\newcommand\RR{{\mathcal R}}
	\newcommand\SSS{{\mathcal S}}
	
	\newcommand\TT{{\mathcal T}}
	\newcommand\ttt{{\mathcal T}_T}
	\newcommand\tT{{\widetilde T}}
	\newcommand\UU{{\mathcal U}}
	\newcommand\VV{{\mathcal V}}
	\newcommand\WW{{\mathcal W}}
	\newcommand\XX{{\mathcal X}}
	\newcommand\YY{{\mathcal Y}}
	\newcommand\ZZ{{\mathcal Z}}
	\newcommand\CH{{\CC\HH}}
	\newcommand\TC{{\TT\CC}}
	\newcommand\EXH{{ \EE (X, \HH )}}
	\newcommand\GXH{{ \GG (X, \HH )}}
	\newcommand\GYH{{ \GG (Y, \HH )}}
	\newcommand\PEX{{\PP\EE  (X, \HH , \GG , \LL )}}
	\newcommand\MF{{\MM\FF}}
	\newcommand\PMF{{\PP\kern-2pt\MM\FF}}
	\newcommand\ML{{\MM\LL}}
	\newcommand\mr{{\RR_\MM}}
	\newcommand\tmr{{\til{\RR_\MM}}}
	\newcommand\PML{{\PP\kern-2pt\MM\LL}}
	\newcommand\GL{{\GG\LL}}
	\newcommand\Pol{{\mathcal P}}
	\newcommand\half{{\textstyle{\frac12}}}
	\newcommand\Half{{\frac12}}
	\newcommand\Mod{\operatorname{Mod}}
	\newcommand\Area{\operatorname{Area}}
	\newcommand\ep{\epsilon}
	\newcommand\hhat{\widehat}
	\newcommand\Proj{{\mathbf P}}
	\newcommand\Hyp{{\mathbf H}}
	\newcommand\bN{\mathbb{N}}
	\newcommand\s{{\Sigma}}
	\renewcommand\P{{\mathbb P}}
	\newcommand\tga{\til{\gamma}}
	\newcommand\til{\widetilde}
	\newcommand\length{\operatorname{length}}
	\newcommand\BU{\operatorname{BU}}
	\newcommand\gesim{\succ}
	\newcommand\lesim{\prec}
	\newcommand\simle{\lesim}
	\newcommand\simge{\gesim}
	\newcommand{\simmult}{\asymp}
	\newcommand{\simadd}{\mathrel{\overset{\text{\tiny $+$}}{\sim}}}
	\newcommand{\ssm}{\setminus}

	\newcommand{\pair}[1]{\langle #1\rangle}

	\newcommand{\inj}{\operatorname{inj}}
	\newcommand{\pleat}{\operatorname{\mathbf{pleat}}}
	\newcommand{\short}{\operatorname{\mathbf{short}}}
	\newcommand{\vertices}{\operatorname{vert}}
	\newcommand{\collar}{\operatorname{\mathbf{collar}}}
	\newcommand{\bcollar}{\operatorname{\overline{\mathbf{collar}}}}

	\newcommand{\tprec}{\prec_t}
	\newcommand{\fprec}{\prec_f}
	\newcommand{\bprec}{\prec_b}
	\newcommand{\pprec}{\prec_p}
	\newcommand{\ppreceq}{\preceq_p}
	\newcommand{\sprec}{\prec_s}
	\newcommand{\cpreceq}{\preceq_c}
	\newcommand{\cprec}{\prec_c}
	\newcommand{\topprec}{\prec_{\rm top}}
	\newcommand{\Topprec}{\prec_{\rm TOP}}
	\newcommand{\fsub}{\mathrel{\scriptstyle\searrow}}
	\newcommand{\bsub}{\mathrel{\scriptstyle\swarrow}}
	\newcommand{\fsubd}{\mathrel{{\scriptstyle\searrow}\kern-1ex^d\kern0.5ex}}
	\newcommand{\bsubd}{\mathrel{{\scriptstyle\swarrow}\kern-1.6ex^d\kern0.8ex}}
	\newcommand{\fsubeq}{\mathrel{\raise-.7ex\hbox{$\overset{\searrow}{=}$}}}
	\newcommand{\bsubeq}{\mathrel{\raise-.7ex\hbox{$\overset{\swarrow}{=}$}}}
	\newcommand{\tw}{\operatorname{tw}}
	\newcommand{\base}{\operatorname{base}}
	\newcommand{\trans}{\operatorname{trans}}
	\newcommand{\rest}{|_}
	\newcommand{\bbar}{\overline}
	\newcommand{\UML}{\operatorname{\UU\MM\LL}}
	\renewcommand{\d}{\operatorname{diam}}
	\newcommand{\hs}{{\operatorname{hs}}}
	\newcommand{\EL}{\mathcal{EL}}
	\newcommand{\tsum}{\sideset{}{'}\sum}
	\newcommand{\tsh}[1]{\left\{\kern-.9ex\left\{#1\right\}\kern-.9ex\right\}}
	\newcommand{\Tsh}[2]{\tsh{#2}_{#1}}
	\newcommand{\qeq}{\mathrel{\approx}}
	\newcommand{\Qeq}[1]{\mathrel{\approx_{#1}}}
	\newcommand{\qle}{\lesssim}
	\newcommand{\Qle}[1]{\mathrel{\lesssim_{#1}}}
	\newcommand{\simp}{\operatorname{simp}}
	\newcommand{\vsucc}{\operatorname{succ}}
	\newcommand{\vpred}{\operatorname{pred}}
	\newcommand\sleft{_{\text{left}}}
	\newcommand\sright{_{\text{right}}}
	\newcommand\sbtop{_{\text{top}}}
	\newcommand\sbot{_{\text{bot}}}
	\newcommand\dpe{d_{pel}}
	\newcommand\de{d_{e}}
	\newcommand\srr{_{\mathbf r}}
	\newcommand\geod{\operatorname{\mathbf g}}
	\newcommand\mtorus[1]{\boundary U(#1)}
	\newcommand\Aleft[1]{\A\sleft(#1)}
	\newcommand\Aright[1]{\A\sright(#1)}
	\newcommand\Atop[1]{\A\sbtop(#1)}
	\newcommand\Abot[1]{\A\sbot(#1)}
	\newcommand\boundvert{{\boundary_{||}}}
	\newcommand\storus[1]{U(#1)}
	\newcommand\Momega{\omega_M}
	\newcommand\nomega{\omega_\nu}
	\newcommand\twist{\operatorname{tw}}
	\newcommand\SSSS{{\til{\mathcal S}}}
	\newcommand\modl{M_\nu}
	\newcommand\MT{{\mathbb T}}
	\newcommand\dw{{d_{weld}}}
	\newcommand\dt{{d_{te}}}
	\newcommand\Teich{{\operatorname{Teich}}}
	\renewcommand{\Re}{\operatorname{Re}}
	\renewcommand{\Im}{\operatorname{Im}}
	\newcommand{\mc}{\mathcal}
	\newcommand{\ccs}{{\CC(S)}}
	\newcommand{\mtdw}{{(\til{M_T},\dw)}}
	\newcommand{\tmtdw}{{(\til{M_T},\dw)}}
	\newcommand{\tmldw}{{(\til{M_l},\dw)}}
	\newcommand{\sxy}{{\Sigma(x,y)}}
	\newcommand{\ts}{{\til{\Sigma}}}
	\newcommand{\tsxy}{{\til{\Sigma}(x,y)}}
	\newcommand{\tmtdt}{{(\til{M_T},\dt)}}
	\newcommand{\tmldt}{{(\til{M_l},\dt)}}
	\newcommand{\trvw}{{\tr_{vw}}}
	\newcommand{\ttrvw}{{\til{\tr_{vw}}}}
	\newcommand{\but}{{\BU(T)}}
	\newcommand{\ilkv}{{i(lk(v))}}
	\newcommand{\pslc}{{\mathrm{PSL}_2 (\mathbb{C})}}
	\newcommand{\tttt}{{\til{\ttt}}}
	\newcommand{\bcomment}[1]{\textcolor{blue}{#1}}
	\newcommand{\sqQ}{{\sqrt{\mathbb Q}}}
	\newcommand{\PSL}{{\mathrm{PSL}_2 (\mathbb{C})}}
	\newcommand{\SL}{{\mathrm{SL}_2 (\mathbb{C})}}
	\newcommand{\pslr}{{\mathrm{PSL}_2 (\mathbb{R})}}
	\newcommand{\pslz}{{\mathrm{PSL}_2 (\mathbb{Z})}}
	\newcommand{\pslq}{{\mathrm{PSL}_2 (\mathbb{Q})}}
	\newcommand{\pglq}{{\mathrm{PGL}_2 (\mathbb{Q})}}
	\newcommand{\pslo}{{\mathrm{PSL}_2 (\OO)}}
	\newcommand{\slq}{{\mathrm{SL}_2 (\mathbb{Q})}}
	\newcommand{\slr}{{\mathrm{SL}_2 (\mathbb{R})}}
	\newcommand{\slz}{{\mathrm{SL}_2 (\mathbb{Z})}}
	\newcommand{\slf}{{\mathrm{SL}_2 (F)}}
	\newcommand{\slo}{{\mathrm{SL}_2 (\mathbb{\OO})}}
	\newcommand{\so}{{\mathrm{SO}}}
	\newcommand{\su}{{\mathrm{SU}}}
	\newcommand{\Gr}{{\mathcal G}}
	\newcommand{\eps}{{\epsilon}}
	\newcommand{\comm}{\operatorname{Comm_\pslr}}
	\newcommand{\comme}{\operatorname{Comm}}
	\newcommand{\commpa}{{\comme_{pa} (\Ga,K)}}
	
	\newcommand{\Star}{\operatorname{star}}
	\newcommand{\jfmchange}[1]{{\color{purple}{#1}}}
	
	\newcommand{\defstyle}[1]{\textbf{#1}}
	\newcommand{\emphstyle}[1]{\emph{#1}}

\newcommand{\aut}{\operatorname{Aut}}
\newcommand{\isom}{\operatorname{Isom}}
\newcommand{\La}{\Lambda}

\newcommand{\lra}{\longrightarrow}
\newcommand{\out}{\operatorname{Out}}
\newcommand{\ol}{\overline}
\newcommand{\qi}{\operatorname{QI}}
\newcommand{\ra}{\rightarrow}
\newcommand{\res}{\mathsf{S}^r}
\newcommand{\restr}{\mbox{\Large \(|\)\normalsize}}
\newcommand{\si}{\sigma}
\newcommand{\stab}{\operatorname{Stab}}
\newcommand{\st}{\textrm{st}}

	\title{Indiscrete Common Commensurators}
	
	\author{Jingyin Huang}
	
	\address{Department of Mathematics, Ohio State University,
		231 W. 18th Ave,	
		Columbus, OH 43210,                                                  
		USA}
	
	\email{huang.929@osu.edu}
	\email{huangjingyin@gmail.com}
	\urladdr{https://sites.google.com/site/huangmath0/}

	\author{Mahan Mj}

	\address{School of Mathematics, Tata Institute of Fundamental Research, Mumbai-40005, India}
	\email{mahan@math.tifr.res.in}
	\email{mahan.mj@gmail.com}
	\urladdr{http://www.math.tifr.res.in/~mahan}

	\subjclass[2010]{20F65, 20F67 }
	\keywords{commensurator, CAT(0) cube complex, Bestvina-Brady group}

	\thanks{Both authors were   supported in part by  the Institut Henri Poincare  (UAR 839 CNRS-Sorbonne Universite), LabEx CARMIN, ANR-10-LABX-59-01, during their participation in the trimester program "Groups acting on fractals, Hyperbolicity and Self-similarity", April-June 2022.
	JH is partially supported by a Sloan fellowship.	MM is supported by  the Department of Atomic Energy, Government of India, under project no.12-R\&D-TFR-5.01-0500,  by an endowment of the Infosys Foundation.
		and by   a DST JC Bose Fellowship. }

	\begin{abstract}
		We develop a framework for common commensurators of discrete subgroups 
		of lattices in isometry groups of CAT(0) spaces. We show that the Greenberg-Shalom hypothesis about discreteness of common commensurators of Zariski dense subgroups and lattices  fails in this generality, even if one imposes strong finiteness conditions. We analyze some examples due to Burger and Mozes in this context and show that they have discrete
		common commensurator.
	\end{abstract}
	
	\maketitle

	\date{\today}

	\tableofcontents
	
	\section{Introduction} A question dating back to Greenberg 
	\cite{greenberg-comm} for $SO(n,1) $ and Shalom \cite{llr-comm} for general semi-simple Lie groups asks the following. We shall assume in this paper that all 
	locally compact groups $G$ have no compact factors.
	
	\begin{qn}\label{qn-shalom}
		Let $G$ be a semisimple Lie group with finite center. Let $K < G$ be a discrete, Zariski-dense subgroup of $G$ of infinite covolume.
	Is the commensurator $\comme_G(K)\subset G$ discrete? In particular, if $\Lambda <G $ is an arithmetic lattice, and $K < \Lambda$ is Zariski-dense, is 
	$\comme_G(K)\subset G$ discrete?
	\end{qn}

Affirmative answers to Question~\ref{qn-shalom} have been obtained in a number of cases
\cite{llr-comm,mahan-commens,kob-mj0,koberda-mj,fmv}. The special case $K < \Lambda < G$ in Question~\ref{qn-shalom} with $\Lambda$ a lattice is the subject of
\cite{shalom-willis,kob-mj0,koberda-mj,fmv}. 
A number of implications of a positive answer to Question~\ref{qn-shalom}, referred to as the Greenberg-Shalom Hypothesis in \cite{bfmv} have been explicated in \cite{bfmv}. 

In this paper, we  extend the scope of Question~\ref{qn-shalom} to the situation 
	$G=Isom(W)$, where $W$ is a finite dimensional CAT(0) cell complex with finitely many cell types:
	
		\begin{qn}\label{qn-shalom1}
		Let $G=Isom(W)$, denote the group of cellular automorphisms of a    finite dimensional CAT(0) cell complex  $W$ with finitely many cell types.
		Let $\Lambda < G$ denote a lattice, and $K < \Lambda$ be \emph{geometrically dense} (in the sense of \cite{caprace-monod1,caprace-monod2}), i.e.\ $K$ leaves
		no proper convex subset of $W$ invariant, nor does $K$ have a fixed point on $\partial W$.
		Is the common commensurator $\comme_G(K)\cap\comme_G(\Lambda) \subset G$ discrete? 
	\end{qn}

	A theorem of Borel \cite{borel} (see Proposition~\ref{commens=commesn}) says that, in the situation of the special case in Question~\ref{qn-shalom}, with $K < \Lambda < G$
	and $\Lambda$ arithmetic,  $\comme_G(K)\subset\comme_G(\Lambda)$ (under mild assumptions), so that the common commensurator equals the commensurator of $K$.
	In the more geometric situation of Question~\ref{qn-shalom1}, Borel's conclusion needs to be added to the hypothesis ahead of time; hence we look at $\comme_G(K)\cap\comme_G(\Lambda)$
	instead of simply $\comme_G(K)$.  In fact, Burger-Mozes \cite[Section 8]{bm-divergence} already provide examples where $\comme_G(K)$ is indiscrete, with $G$ the automorphism
	group of a tree, and $K$ the commutator subgroup of a lattice $\Lambda$ in $G$. However, it turns out that in this case, $\comme_G(K)$ is  not contained in $\comme_G(\Lambda)$;  we will in fact show that $\comme_G(K)\cap\comme_G(\Lambda) \subset G$ is discrete. In particular, the example in \cite[Section 8]{bm-divergence} does not provide 
	a counterexample to Question~\ref{qn-shalom1}. There is  a further elementary reason for  considering  $\comme_G(K)\cap\comme_G(\Lambda)$ instead of  just $\comme_G(K)$ as we explain at the end of this introduction.
	

\medskip
	
	In this paper,  we provide a 
	 number of counterexamples  to Question~\ref{qn-shalom1}.  These examples may be arranged to have additional finiteness properties. 
	 	We start with a negative answer to Question~\ref{qn-shalom1} for $W$ a tree and $G$ its automorphism group (see Theorem~\ref{thm-wisecsc}
	 	for a precise statement).


	\begin{theorem}
		\label{thm:main1}
	There exists a locally finite regular tree $T$, a lattice $L$ in $G=\aut(T)$, and an infinite index infinite normal subgroup $K$ of $L$ such that 
	$$\comme_G(K)\cap \comme_G(L)$$ is not discrete in $G$. Equivalently, 
	$K, L$ have an indiscrete common commensurator in $\aut(T)$.
	\end{theorem} 

 That $K$ and its relatives below are geometrically dense follows easily from its construction and \cite[Theorem 1.10]{caprace-monod1}.
The examples in Theorem~\ref{thm:main1} are built from
  examples due to Wise of irreducible lattices  acting on 
  a product of tress as follows.
\begin{eg}
In \cite[Section 4]{wise-csc}, Wise constructed an irreducible lattice $\Gamma$ acting on a product of two trees $T_h\times T_v$ such that the quotient of $T_h$ under the factor action of $\Gamma$ is a circle with a single vertex. In particular this gives an HNN-extension structure to $\Gamma$ by considering $\Gamma\acts T_h$. Specifically, one has $\Gamma=(F_3)\ast_{F_5}$. Let $\phi: \Gamma \to \Z$ denote the homomorphism
induced by the HNN extension description of $\Gamma$ and let $\Gamma_0=\ker (\phi)$. Let $Y$ denote the cyclic cover of $X$ corresponding to $\Gamma_0$.

Take a vertex $x\in T_v$ and let $L$ be the stabilizer of $T_h\times \{x\}$. Let $K=\ker (\phi|_L)$. Then $K$ is an infinite and infinite index normal subgroup of $L$. We treat $L$ as a lattice in $G=\aut(T_h)$. We show in Theorem~\ref{thm-wisecsc} (stated as Theorem~\ref{thm:main1} above) that 	$\comme_G(K)\cap \comme_G(L)$ is not discrete. In particular, 
$K$ has indiscrete commensurator in $\aut(T_h)$.
\end{eg}

In Proposition~\ref{prop-startingpoint}, we  considerably generalize this example. We start with an irreducible lattice acting on the product of two CAT(0) complexes with finitely many isometry types of cells, satisfying some additional
sufficient conditions. Under such conditions, one obtains a strong negative answer to  Question~\ref{qn-shalom1}. Now we mention an application of Proposition~\ref{prop-startingpoint}.

	
	In Theorem~\ref{thm:main1}, $K$ is infinitely generated. We investigate
	what happens if we impose stronger finiteness properties on $K$, i.e.\
	whether such conditions could possibly
	 lead to a positive answer to Question~\ref{qn-shalom1} for CAT(0) spaces.
	 We show that this is not the case.
	  We  give below a considerably more sophisticated family  of examples,
	where 
	\begin{enumerate}
		\item $T_h$ is replaced by the universal cover of a Salvetti complex
		corresponding to a right-angled Artin group (RAAG), and
\item $K$ has strong finiteness properties.		
	\end{enumerate} 
	 In fact, $K$ is a Bestvina-Brady group of a different RAAG. More precisely (see Theorem~\ref{thm-bb} and Corollary~\ref{cor-bb}),
	\begin{theorem}\label{thm-bb-intro}
		Let $\GG$ be a finite simplicial graph, with a base vertex $v$ and a vertex $w\neq v$. Let $\LL$ be a wedge of four copies of $\GG$ along $v$. Let $X_\LL$ be the universal cover of the Salvetti complex of $\LL$. Let $BB_\LL$ and $A_\LL$ be the Bestvina-Brady subgroup and the right-angled Artin group associated with $\LL$, viewed as discrete subgroups of $G=\aut(X_\LL)$. Then
		$\comme_G(BB_\LL)\cap \comme_G(A_\LL)$ is not discrete.
	\end{theorem}

Combining the above theorem with the computation of finiteness properties of Bestvina-Brady groups in \cite{bb-morse} gives the following, which is in sharp contrast with the case that $K$ is not finitely generated in Theorem~\ref{thm:main1}.  
\begin{cor}
For each $n\ge 1$, there exists a CAT(0) cube complex $X$, a uniform lattice $L$ in $G=\aut(X)$, and an infinite index infinite normal subgroup $K$ of $L$ such that $K$ is of type $F_n$ and 
	$\comme_G(K)\cap \comme_G(L)$ is not discrete in $G$.
\end{cor}

\begin{rmk} We do not know the answer to the following question. Let 
	$\mathcal G$ be a finite  graph  which is not a complete graph.
	Let $G=\aut(X_{\mathcal G})$.  Is
 	$\comme_G(BB_\GG)\cap \comme_G(A_\GG)$  non-discrete in general?
 	
 	  If we  glue several copies of $\GG$ along a distinguished vertex $v\in \GG$ to form a new graph $\mathcal L$, then this gives non-trivial symmetries of $\mathcal L$. These symmetries can be used to produce many new symmetries of $X_{\mathcal L}$. So, $\aut(X_{\mathcal L})$ will be ``much larger'' than $\aut(X_{\mathcal G})$. This increase in the symmetry group is used in  the proof of Theorem~\ref{thm-bb-intro}. The idea of enlarging $\aut(X_{\mathcal G})$ by gluing multiple copies of $\mathcal G$ is originally due to Hughes in a somewhat different context \cite{fibering}.
\end{rmk}

	
	Subsequently, we analyze a set of examples due to  Burger and Mozes \cite[Section 8.1]{bm-divergence}. They give an example of a normal subgroup $K$ of a tree lattice $\Gamma$,
	so that $\Gamma$ is a free group on $d$ generators acting cocompactly on a $2d-$regular tree $T_{2d}$. While Burger and Mozes show that $K$ has dense commensurator in $G=\aut(T_{2d})$, we show the following (see Proposition~\ref{prop-disccomm}).
	
	 \begin{prop}\label{prop-intro2}
		With $G=\aut(T_{2d})$, and $ \Gamma, K_d$ as in the Burger-Mozes example above, the common commensurator $\comme_G(K_d) \cap \comme_G(\Gamma)$ is discrete.
	\end{prop}
	Thus, the Burger-Mozes examples end up providing a counterexample (in the  general context of automorphism groups of trees) to
	a Theorem due to Borel (see Theorem \ref{commens=commesn}) that says that
	the commensurator of a Zariski dense subgroup of an arithmetic lattice is contained in the commensurator of  the ambient lattice.

The Burger-Mozes examples in \cite[Section 8.1]{bm-divergence} do in fact provide instances where $\comme_G(K)$ is indiscrete and further $[\comme_G(K): N_G(K)]$ is infinite, where $N_G(K)$ denotes the normalizer of $K$ in $G$. Our focus in this paper is to bring out a subtler contrast with the arithmetic situation beyond that covered by Borel's theorem \cite{borel}. Hence we work the conclusion of Borel's theorem into the hypothesis and look at the common commensurator 
$\comme_G(K)\cap\comme_G(\Lambda)$.\\


\noindent {\bf Commensurator versus normalizer:} We conclude this introductory section with a further comment about the  way Question~\ref{qn-shalom1} is formulated. 
Apart from Borel's theorem \cite{borel},	a further reason for looking at the common commensurator $\comme_G(K)\cap\comme_G(\Lambda)$ instead of simply the  commensurator $\comme_G(K)$ in Question~\ref{qn-shalom1} is as follows. 

Again, take $G = \aut(T)$ as an example.
	An uninteresting way for $K$ to have an indiscrete commensurator in $G$ is  that $K$
	(or some finite index subgroup of $K$) already has
	an indiscrete normalizer $N_G(K)$ in $G=\aut(T)$. In the special case that $K < \Lambda$ is a normal subgroup, 
	this is equivalent to saying that the automorphism group of the graph $T/K$ itself
	is indiscrete. Such examples are relatively easy to find, e.g.\ when $T/K$
	itself is a Cayley-Abels graph of an indiscrete tdlc group.
	The examples in \cite{bm-divergence} as well as those described in this
	paper are fundamentally different from this class of  examples.

\section{Complete Square complexes and friends}\label{sec-csc}
\subsection{Complete square complexes and graph of graphs}\label{subsec-csc}
\begin{defn}\label{def-csc}\cite{wise-csc} A \emph{complete square complex (CSC)} is a non-positively curved square complex such that the 
link of every vertex is a complete bipartite graph $K_{m,n}$.
\end{defn}
It was shown by Wise in \cite{wise-csc} that a CAT(0) square complex is a CSC if and only if its universal cover is a product of trees $T_h \times T_v$.

One way to construct a CSC, discussed in \cite{wise-csc}, is the following. Let $\GG$ be a finite graph. Let $X$ be a graph of spaces with  underlying graph $\GG$, such that 
\begin{enumerate}
\item edge  and vertex space are finite graphs,
\item any boundary morphism from an edge space to a vertex space is a covering map of finite degree.
\end{enumerate} 
 Then $X$ is a CSC \cite{wise-csc}. 

Let $\Gamma=\pi_1(X)$. Then the universal cover of $X$ is a product of two trees $T_h\times T_v$, and $\Gamma$ is a cocompact lattice in  $\aut(T_h) \times \aut(T_v)$. 


Conversely, given a group $\Gamma$ acting on $T_h\times T_v$ freely without exchanging the two factors, the quotient CSC, denoted $X$, admits  two graph of graphs structures with boundary morphisms being covering maps of finite degree.  After subdividing edges of $T_h$ and $T_v$ if necessary, we may assume in addition that the factor actions $\Gamma\acts T_h$ and $\Gamma\acts T_v$ do not have edge inversions.
The quotient map
$\Pi: T_h  \times T_v  \to X$ sends each $\{x\} \times T_v$ to a \emph{vertical fiber}
of $X$ through ${\Pi(x)}$ and each $T_h\times \{y\}$ to a \emph{horizontal fiber}
of $X$ through ${\Pi(y)}$. The quotient space of $X$ obtained by quotienting each
vertical fiber (resp.\ horizontal fiber) to a point gives a \emph{horizontal 
	base graph $\HH$} (resp.\  a \emph{vertical 
	base graph $\VV$}). Note that $\HH = T_h/\Gamma$ and  $\VV = T_v/\Gamma$.
Then we have natural maps $P_h: X \to \HH$ and 
$P_v: X \to \VV$, where 
\begin{enumerate}
	\item the pre-image of each point $z$ (in $\HH$ or $\VV$) under  $P_h: X \to \HH$ or
	$P_v: X \to \VV$ is a graph, denoted
	$\FF_z$. We shall refer to   $\FF_z$ as the \emph{fiber over $z$}. If $z \in \HH$,
	(resp.\ $z \in \VV$) 
	$\FF_z$ is a vertical (resp.\ horizontal) fiber.
	\item over each open edge $e$  (in $\HH$), the bundle $P_h: P_h^{-1}(e) \to e$
	is trivial, so that we may write $P_h^{-1}(e) = \FF_e \times e$. Here, $ \FF_e $ is the
	fiber over any point in $e$. A similar description holds for $\VV$.
	\item If $z$ is a vertex of the graph $\HH$ and is a boundary point of $e$, then 
	$\FF_e$ is a finite-sheeted cover of $\FF_z$. A similar description holds for 
	vertices of $\VV$.
\end{enumerate}
There is a dual point of view that regards vertical (resp.\ horizontal) fibers as \emph{multi-sections} of 
$P_v: X \to \VV$   (resp.\ $P_h: X \to \HH$). 

\subsection{More properties of vertical fibers and horizontal fibers}

We will work in a slightly more general setting in this subsection.
 Let $\KK_v, \KK_h$ be piecewise Euclidean CAT(0) polyhedral complexes with finite shape (in the sense of \cite[Chapter I.7]{brid-h}), i.e.\ there are finitely many isometry types of cells.
Further
 suppose that $\Gamma$ acts freely and cocompactly by isometries 
on $\KK_h \times \KK_v$ such that
\begin{enumerate}
	\item $\Gamma$ does not exchange the two factors;
	\item the factor actions $\Gamma\acts \KK_h$ and $\Gamma\acts \KK_h$ have no inversions, i.e. whenever a closed cell is fixed setwise, then it is fixed pointwise.
\end{enumerate}
Let $X = (\KK_h \times \KK_v)/\Gamma$. Then the  discussion in Section~\ref{subsec-csc}  generalizes in a straightforward way to give the following. There exists a horizontal
(resp.\ vertical) base	polyhedron complex $\HH$ (resp.\ $\VV$) and a map
$P_h: X \to \HH$ (resp.\ 
$P_v: X \to \VV$) such that
\begin{enumerate}
	\item The vertical fibers $P_h^{-1}(z)$, $z \in \HH$ (resp.\ 
	horizontal fibers $P_h^{-1}(z)$, $z \in \VV$) are polyhedral complexes.
	\item Each vertical fiber $P_h^{-1}(z)$, $z \in \HH$ is a compact
	quotient of $\KK_v \times \{w\}$ for some $w \in \KK_h$ by the stabilizer
	(in $\Gamma$) of $\KK_v \times \{w\}$. Similarly, 
	each horizontal fiber $P_v^{-1}(z)$, $z \in \VV$ is a compact
	quotient of $\{w\}\times \KK_h  $ for some $w \in \KK_v$ by the stabilizer
	(in $\Gamma$) of $\{w\}\times \KK_h  $.
	\item Each vertical fiber $P_h^{-1}(z)$ is the unique \emph{multi-section} of 
	$P_v$ through any $x \in P_h^{-1}(z)$. Similarly, 
	each horizontal fiber $P_v^{-1}(z)$ is the unique \emph{multi-section} of 
	$P_h$ through any $x \in P_v^{-1}(z)$.
	\item $\HH = \KK_h/\Gamma$ and  $\VV = \KK_v/\Gamma$.
\end{enumerate} 

Note that there are natural maps $P_h^{-1}(z)\to \mathcal V$ and $P_v^{-1}(z)\to \mathcal H$. The  inverse images of points under these maps have finite cardinality. That is why we call them multi-sections over $\mathcal V$ or $\mathcal H$.

\begin{rmk}\label{rmk-salvxtree}
Later	in the paper, we shall be particularly concerned with the case that $\KK_h$ is the
	universal cover of a Salvetti complex and $\KK_v$ is a tree.
\end{rmk}

\section{A criterion for non-discrete commensurators and first examples}

\subsection{The criterion}
Much of this section and the next is devoted to guaranteeing the hypotheses of Proposition~\ref{prop-startingpoint} below, which gives a criterion for nondiscrete commensurators. We therefore spell this out explicitly.

Recall that $\Gamma$ is  \emph{irreducible as an abstract group} if $\Gamma$ does not have a finite index subgroup $\Gamma'$ admitting a splitting $\Gamma'=\Gamma_1\times \Gamma_2$ with $\Gamma_i$ being infinite for $i=1,2$. 

\medskip

\noindent {\bf Hypotheses and notation for Proposition~\ref{prop-startingpoint}:}
\begin{enumerate}
	\item $\KK_h$ and $\KK_v$ are finite dimensional 
	proper CAT(0) polyhedral complexes with finite shape.
	\item $\Gamma$ acts freely, cocompactly and properly discontinuously by isometries on $\KK_h \times \KK_v$ without exchanging the two factors. Let
	$X = (\KK_h \times \KK_v)/\Gamma$.
	\item The homomorphism $\Gamma\to \aut(\KK_h)$ induced by the factor action $\Gamma\acts \KK_h$ has indiscrete image in $\aut(\KK_h)$. This holds, for instance, when $\Gamma$ is irreducible as an abstract group \cite{caprace-monod2}.
	\item The fundamental group of $\HH = \KK_h/\Gamma$ admits a surjective homomorphism $\beta:\pi_1(\HH)\to L$ onto some infinite group $L$. Let $\HH_h$ denotes the cover of $\HH$ corresponding to $\ker \beta$.
	\item $\phi : \Gamma \to L$ be the composition of $\Gamma\to \pi_1(\HH)$ (induced by the projection map $X\to \HH$) and $\beta$. Let $\ker \phi=\Gamma_0$ and $Y= (\KK_h \times \KK_v)/\Gamma_0$ denote the cover of $X$
	corresponding to $\Gamma_0$.
	\item $M$ denotes a horizontal fiber passing through a vertex of $X$.
\end{enumerate}

Note that $\pi_1 M$ acts by deck transformations on the universal cover
$\mathcal K_h$ of $M$. Thus we will view $\pi_1 M$ as a subgroup of $\aut(\mathcal K_h)$. Then $\pi_1 M$ is a cocompact lattice in $\aut(\mathcal K_h)$. On the other hand, the inclusion map $M\to X$ is $\pi_1$-injective. So we also view $\pi_1 M$ as a subgroup of $\Gamma$.

\begin{prop}\label{prop-startingpoint} With notation and hypotheses as above,
	let $K=\ker (\phi|_{\pi_1 M})$ denote the kernel of $\phi:\Gamma\to L$ restricted to $\pi_1 M\le \Gamma$. Now we view $K\le \pi_1 M$ as subgroups of $G=\aut(\KK_h)$. Then 
	\begin{enumerate}
		\item $K$ is of infinite index in $\pi_1 M$.
		\item 	$\comme_G(K)\cap \comme_G(\pi_1 M)$ is indiscrete in $G$, in particular, $K$ has indiscrete commensurator in $\aut(\KK_h)$.
		\item  If $L\cong \mathbb Z$, then $\comme_G(K)$ does not have a finite index subgroup which normalizes some finite index subgroup of $K$.
	\end{enumerate}
\end{prop}

\begin{proof}
	For Assertion (1), we first prove the claim that  $\phi$ is surjective. Indeed, it suffices to show that $X\to \HH$ is $\pi_1$-surjective. Consider the factor action of $\Gamma$ on $\KK_h$ and a basepoint $x\in \KK_h$. Each $\gamma\in \Gamma$ gives a loop in $\HH$ as follows.  
Let $\alpha$ be a path in $\KK_h$ from $x$ to $\gamma x$.
Then the	image of $\alpha$
 under the quotient map $\KK_h\to \KK_h/\Gamma=\HH$ is a loop in $\HH$. As $\KK_h$ is simply-connected, the homotopy class of the resulting loop does not depend on the choice of the path
 $\alpha$. This gives a homomorphism $\Gamma\to \pi_1(\HH)$ which is an alternative description of the induced map on $\pi_1$ for 
 the projection $X\to \HH$. Since the map $\KK_h\to \KK_h/\Gamma=\HH$ has 
 the path lifting property,  $\Gamma\to \pi_1(\HH)$ is surjective.
	
	Let $N$ be a vertex stabilizer of $\Gamma\acts \KK_h$. Then $N\subset \ker \phi$, and there is an $N$-invariant vertical slice of $\KK_h\times \KK_v$ where $N$ acts cocompactly. As $\pi_1 M$ can be viewed as a vertex stabilizer of $\Gamma\acts \KK_v$, there is a $\pi_1 M$-invariant horizontal slice of $\KK_h\times \KK_v$ where $\pi_1 M$ acts cocompactly. If $K$ is of finite index in $\pi_1 M$, then $K$ acts cocompactly on the same horizontal slice. Then $\ker \phi$ contains the subgroup generated by $K$ and $N$, and this subgroup acts cocompactly on $\KK_h\times \KK_v$. Thus $\ker\phi$ is finite index in $\Gamma$, which contradicts the previous paragraph and
	the hypothesis that $L$ is infinite.\\
	
	For Assertion (2), consider the factor action $\Gamma\acts \KK_v$.
	 Then there is a vertex $x\in \KK_v$, such that $\pi_1 M$ is the stabilizer of $x$. As $\Gamma$ acts geometrically on $\KK_h\times \KK_v$, we know $\KK_v$ is locally finite. By considering the action of $\Gamma$ on the 1-skeleton of $\KK_v$ and applying Lemma~\ref{lem:commensurated} below, we know that $\pi_1 M$ is commensurated by $G$. 
	On the other hand, if we restrict the action $\Gamma\acts \KK_v$ to $\Gamma_0\acts \KK_v$, then $K$ is the $\Gamma_0$-stabilizer of $x$. Thus $K$ is commensurated by $\Gamma_0$ by Lemma~\ref{lem:commensurated} again. Let $\Gamma'$ be the subgroup of $\Gamma$ generated by $\pi_1 M$ and $\Gamma_0$. As $K$ is normal in $\pi_1 M$, $K$ is commensurated by $\pi_1 M$. Thus $K$ is commensurated by $\Gamma'$.
	
	Let $\Phi:\Gamma\to \aut(\KK_h)$ be the homomorphism induced by the factor action. Then $\Phi_{|\pi_1 M}$ is injective. As $\Gamma'$ commensurates both $\pi_1 M$ and $K$, we know 
	that $\Phi(\Gamma')$ commensurates both $\Phi(\pi_1 M)$ and $\Phi(K)$.
	Note that $\Gamma'$ acts cocompactly on $\KK_h\times \KK_v$ as the factor action $\pi_1 M\acts \KK_h$ is cocompact and the factor action $\Gamma_0\acts \KK_v$ is cocompact. Thus $\Gamma'$ is of finite index in $\Gamma$. It follows that $\Phi(\Gamma')$ is of finite index in $\Phi(\Gamma)$. Since $\Phi(\Gamma)$ is not discrete by assumption, neither is $\Phi(\Gamma')$. The second assertion follows.\\
	
	For Assertion (3), we argue by contradiction. Suppose $\Gamma_0$ has a finite index subgroup $\Gamma'_0$ such that $\Gamma'_0$ normalizes a finite index subgroup $K'$ of $K$. Consider the factor action $\Gamma'_0\acts \KK_v$ and let $K''$ be the kernel of this action.
	As $K$ is a vertex stabilizer for the factor action $\Gamma_0\acts \KK_v$, our assumption implies that $K''$ is of finite index in each vertex stabilizer of the action $\Gamma'_0\acts \KK_v$.  Let $N_{x_0}$ be the stabilizer of a vertex $x_0\in \KK_h$ with respect to the other factor action $\Gamma'_0\acts \KK_h$. For $g\in N_{x_0}$, let $\bar g\in \aut(\KK_h)$ be the action of $g$ on the horizontal factor. As $K''$ is normal in $\Gamma'_0$, we know
	that $\bar g K''(\bar g)^{-1}=K''$ in $\aut(\KK/h)$. Thus $\bar g$ descends to a cellular automorphism of the orbit space $\KK_h/K''$ fixing $\bar x_0\in \KK_h/K''$, where $\bar x_0$ is the image of  $x_0$ under $\KK_h\to \KK_h/K''$. This gives a group automorphism $\Theta:N_{x_0}\to \aut(\KK_h/K'',\bar x_0)$. 
	
 The assumption $L\cong \mathbb Z$ and Assertion (1) implies that $\KK_h/K$ is an infinite cyclic cover of $M$. Thus $\KK_h/K$ is quasi-isometric to $\mathbb Z$.	The previous paragraph implies that $K''$ is finite index in $K$. Thus $\KK_h/K''$ is quasi-isometric to $\mathbb Z$. As $N_{x_0}$ is finitely generated, Lemma~\ref{lem:finite} below implies that $\im \Theta$ is finite. Let $N'_{x_0}=\ker \Theta$, and let $\Gamma''_0$ be the semi-direct product of $K''$ and $N'_{x_0}$. Then $(\KK_h\times \KK_v)/\Gamma''_0$ has the structure of a fiber bundle, with base being $\KK_h/K''$ and fiber being $\KK_v/N'_{x_0}$. As $N'_{x_0}$ is of finite index in $N_{x_0}$, the fibers are compact. As the monodromy of this fiber bundle respects the cell-structure of $\KK_v/N'_{x_0}$, we know that the monodromy representation $\pi_1(\KK_h)/K''\to \aut(\KK_v/N'_{x_0})$ has finite image as $\aut(\KK_v/N'_{x_0})$ is finite. Thus $(\KK_h\times \KK_v)/\Gamma''_0$ has a finite cover which is a product. Hence $N'_{x_0}$ has a finite index subgroup acting trivially on $\KK_h$, contradicting Hypotheses (3) of Proposition~\ref{prop-startingpoint}.
\end{proof}

The proof of Proposition~\ref{prop-startingpoint} used Lemmas~\ref{lem:finite}
and \ref{lem:commensurated} below. We now provide the proofs of these statements. 


\begin{lemma}
	\label{lem:finite}
Let $\mathcal{G}$ be a  locally finite graph with an upper bound on valence admitting
 a  proper cocompact $ \Z-$action by graph automorphisms. Let $\Gamma$ be a finitely generated group in the automorphism group of $\mathcal {G}$. Suppose $\Gamma$ fixes a vertex of $\mathcal G$. Then $\Gamma$ is finite.
\end{lemma}

\begin{proof}
After a subdivision of $\mathcal G$ if necessary, we can assume  that
the $\Z-$action on $\mathcal G$ is simplicial. Let $V$ be the vertex set of $\mathcal G$. As $\mathcal G$ is quasi-isometric to $\Z$, we can find a bijective quasi-isometry $q:V\to \mathbb Z$. Then $q$ conjugates the isometric action $\aut(\mathcal G)\acts V$ to an action $\aut(\mathcal G)\stackrel{\rho_0}{\acts} \mathbb Z$  by $(L,A)$-quasi-isometries for some $L\geq 1$ and $A\geq 0$. By \cite[Proposition 1.13]{MR3761106}, there is an isometric action $\aut(\mathcal G)\stackrel{\rho_1}{\acts}\mathbb Z$ and a surjective $\aut(\mathcal G)$-equivariant $(L',A')$-quasi-isometry $f:\Z\to \Z$. Since $\rho_1:\Gamma\acts \mathbb Z$ is an isometric action and the $\Gamma-$action on $\GG$ has a fixed point, then, after passing to a finite index subgroup if necessary, we can assume that the action $\Gamma\stackrel{\rho_1}{\acts}\mathbb Z$ is trivial. For $n\in \mathbb Z$, let $G_n$ be the permutation group of $f^{-1}(n)$. Then we have an injective homomorphism $\Gamma\to \prod_{n\in \mathbb Z} G_n$. As there is a uniform upper bound for the cardinality of $f^{-1}(n)$, we know there are only finitely many isomorphism types of $G_n$. As $\Gamma$ is finitely generated, for each $G_n$, there are only finitely many homomorphisms from $\Gamma$ to $G_n$. Thus the kernel of $\Gamma\to \prod_{n\in \mathbb Z} G_n$ is the intersection of finitely many finite index subgroups of $\Gamma$. As $\Gamma\to \prod_{n\in \mathbb Z} G_n$ is injective, we know $\Gamma$ is finite.
\end{proof}	

The following is well-known, we provide a proof for the convenience of the reader.

\begin{lemma}
	\label{lem:commensurated}
	Suppose $G$ is a group acting on a locally finite connected graph $\GG$ by automorphisms. Then each vertex stabilizer of $G$ is commensurated by $G$.
\end{lemma}

\begin{proof}
	Given a vertex $v\in \GG$ with an edge $e\subset \GG$ containing $v$, $\stab_G(v)$ contains $\stab_G(e)$ as a finite index subgroup, as there are only finitely many edges containing $v$. Then for any two adjacent vertices, their stabilizers are commensurable. Now the lemma follows from the connectedness of $\GG$. 
\end{proof}

\subsection{Some first examples}
While checking the   hypotheses of Proposition~\ref{prop-startingpoint} in full generality will occupy the next few sections, we furnish one example right away.

\begin{lemma}[\cite{wise1996non}]
	There exists a compact CSC, denoted $X$, with a graph of graphs structure such that 
	\begin{enumerate}
		\item the underlying graph $\GG$ is a circle;
		\item $\Gamma=\pi_1(X)$ is an irreducible lattice acting on $\widetilde X=T_h\times T_v$.
	\end{enumerate}
\end{lemma}

\begin{proof}
	We will focus on particular examples of graph of graphs as above such that the underlying graph $\GG$ is a circle, constructed by Wise in  \cite{wise1996non}, see also \cite[Section 4]{wise-csc}. In this case the underlying graph $\GG$ is a circle with a single vertex. The vertex space is a wedge of three circles, the edge space is a graph with two 0-cells, with a loop at each and 4 edges connecting the two 0-cells. The two boundary morphisms from the edge space to its vertex space are two different degree two covering maps. In particular, $\pi_1(X)$ is isomorphic to an HNN extension $(F_3)\ast_{F_5}$, where $F_5$ includes into $F_3$ as two distinct index two subgroups. It is shown in \cite{wise1996non} that $\Gamma=\pi_1(X)$ is irreducible.
	
	More generally, we can start with an irreducible lattice $\Gamma_0$ in a product of trees as in the work of
	Burger-Mozes \cite{bm97,bm2000}, Wise \cite{janzen2009smallest,wise-aperiodic,wise-csc}, Janzen-Wise \cite{janzen2009smallest}, given as a free product
	with amalgamation (and not an HNN extension). 
	Thus, $\Gamma_0=A\ast_C A'$, and the underlying CSC (denoted $X$) is a graph of graphs over an edge, with two vertex spaces being $Y_1$ and $Y_2$. We can glue two identical copies of $X$ isometrically along $Y_1\sqcup Y_2$ to obtain a new CSC, denoted $X'$. Then the underlying graph of $X'$ is circle. If $\pi_1 X$ is an irreducible lattice acting on product of two tress, then $\pi_1 X'$ is also an irreducible lattice acting on (possibly different) product of trees. 
\end{proof}


\begin{theorem}\label{thm-wisecsc}
Let $\Gamma=(F_3)\ast_{F_5}$ denote the fundamental group of the CSC complex described by Wise in \cite[Section 4]{wise-csc} so that $\Gamma$ acts on a product $T_h \times T_v$ of regular trees as a lattice. Let $\phi: \Gamma \to \Z$ denotes the $\pi_1$-homomorphism induced by $X=(T_h\times T_v)/\Gamma\to T_h/\Gamma\cong \mathbb S^1$ and let $\Gamma_0=\ker (\phi)$.

Let $M$ be a horizontal fiber of $X$ passing through a vertex of $X$, and we view $\pi_1 M$ as a subgroup of $\Gamma$, as well as a subgroup of $\aut(T_h)$. Let $K=\ker (\phi|_{\pi_1 M})$. Then the following hold true.
\begin{enumerate}
	\item The common commensurator of $K$ and $\pi_1 M$   in $\aut(T_h)$ is indiscrete.
	\item The commensurator of $K$ in $\aut(T_h)$ does not have a finite index subgroup that normalizes some finite index subgroup of $K$.
\end{enumerate}
\end{theorem}

\begin{proof}
This follows directly from Proposition~\ref{prop-startingpoint} by setting
$T_h =\KK_h$.
\end{proof}

\section{Higher dimensional examples and finiteness properties}\label{sec-hd}
\subsection{RAAGs and Salvetti complexes}\label{sec-prel}
We refer to \cite{charney2007introduction} for background on right-angled Artin groups and Salvetti complexes.
Let $\Gamma_\GG$ be the right-angled Artin group associated with the finite graph $\GG$. Let $S_\GG$ be the Salvetti complex of $\Gamma_\GG$, and let $X_\GG$ be the universal cover of $S_\GG$. Recall that the 1-skeleton of $X_\GG$ is the Cayley graph of $\Gamma_\GG$. Hence each edge in the 1-skeleton of $X_\GG$ is oriented and labeled by the label of a vertex of $\GG$.

Each complete subgraph $\Delta$ of $\GG$ gives a torus $S_\Delta$ inside $S_\GG$, called a \emph{standard torus} of type $\Delta$. A \emph{standard flat} of type $\Delta$ in $X_\GG$ is a lift of a standard torus of type $\Delta$. 
A \emph{standard line} is a 1-dimensional standard flat. An automorphism of $X_\GG$ is \emph{flat-preserving}, if it sends standard flats to standard flats, or equivalently, it sends standard lines to standard lines.

Given a group $\Gamma$, a normal subgroup $N$ of $\Gamma$ is \emph{non-trivial} if $N$ is infinite and has infinite index in $\Gamma$.

Our goal here is to construct an irreducible lattice $H$ acting on $X_\GG\times T$ for some graph $\GG$ and some regular tree $T$ such that the following holds:
\begin{enumerate}
	\item For a vertex $x\in T$, let $H_x$ denote the $H$-stabilizer  of the slice $X_\GG\times \{x\}$, viewed as a subgroup of $\aut(X_\GG)$ via the action of $H_x$ on the first factor $X_\GG$. We demand that there exists
	a vertex $x\in T$, such that $H_x$  admits a non-trivial normal subgroup $N\unlhd H_x$ with non-discrete commensurator in $\aut(X_\GG)$.
	\item $N$ satisfies higher finiteness properties. 
\end{enumerate} 
This will allow us to generalize the normal subgroup $K$ in the example of Theorem~\ref{thm-wisecsc} to a normal subgroup $N$ satisfying higher finiteness properties. Note that the example $K$ in Theorem~\ref{thm-wisecsc} is typically infinitely generated, and hence satisfies no finiteness properties.\\

\noindent {\bf Scheme:} 
In order to arrange this, we  start with a ``seed group'', which is an irreducible lattice $H_0$ acting on a product of two trees. 
We then enlarge the first tree factor to a carefully chosen $X_\GG$ and extend the action of $H_0$ to arrange the desired irreducible lattice acting on $X_\GG\times T$. Our first step will be Lemma~\ref{lem:tree} below: the lattice in the lemma is required to be irreducible, so that after the extension, we  still obtain an irreducible lattice. Irreducibility ultimately leads to non-discrete commensurators in the example we want to arrange.  Lemma~\ref{lem:tree} has two requirements:
\begin{enumerate}
	\item The first  ensures that we can extend the action to a larger space.
	\item The second gives us control on a particular slice. This will ultimately be helpful for proving finiteness properties of normal subgroups of stabilizers of certain slices.
\end{enumerate}  

\subsection{A seed irreducible lattice}\label{sec-seed}
Let $T_h$ be the Cayley graph of the free group with $4$ generators $F_4$, generated by $w_1,w_2,w_3,w_4$. 
Note that edges in $T_h$ are equipped with the usual labels and orientations of edges
as in a standard Cayley graph.

Let $T_v$ be another copy of the Cayley graph of $F_4$, generated by $u_1,u_2,u_3,u_4$. Then each edge of $T_h\times T_v$ is oriented and labeled. We will think $T_h$ as the ``horizontal factor'', and $T_v$ as the ``vertical factor''.

The goal of this subsection is to construct an irreducible lattice acting on $T_h\times T_v$  respecting the labels and orientations of edges only up to an extent. This is made precise as follows.

\begin{lemma}
	\label{lem:tree}
	There exists a torsion free irreducible lattice $H$ acting on $T_h\times T_v$ without exchanging the two factors such that
	\begin{enumerate}
		\item the action of $H$ preserves orientations (but not necessarily labels) of edges, and sends standard flats to standard flats;
		\item there exists a vertex $x\in T_v$ such that the action of the stabilizer of $T_h\times \{x\}$ on it is also label-preserving.
	\end{enumerate}
\end{lemma}

To construct such an example, we first start with the action of a torsion free group $L$ on $U_1\times U_2$ where $U_1, U_2$ are both regular trees of valence $4$. To start off, edges of $U_1$ and $U_2$ are not oriented and labeled. For a vertex $z\in U_2$, we define $U_{1,z}$ to be the quotient of $U_1\times \{z\}$ by the $L$-stabilizer of $U_1\times \{z\}$. We now formulate a condition on $L$ such that one can use $L$ as a building block to construct the desired example in Lemma~\ref{lem:tree}.

\begin{lemma}
	\label{lemma:smaller tree}
	Suppose that there exists a torsion free irreducible lattice $L$ acting on $U_1\times U_2$ properly and cocompactly by automorphisms without exchanging the two factors such that one can find a vertex $z_0\in U_2$ with the following properties: 
	\begin{itemize}
		\item each edge in $U_{1,z_0}$ is embedded;
		\item it is possible to label each edge of $U_{1,z_0}$ by one of $\{w_1,w_2,w_3,w_4\}$ such that the labels of two different edges incident on the same vertex of $U_{1,z_0}$ are different.
	\end{itemize}
	Then Lemma~\ref{lem:tree} holds true.
\end{lemma}

\begin{proof}
	We pull back edge labels on $U_{1,z_0}$ to obtain  edge labels on $U_1$. Moreover, we assign  arbitrary edge labels on $U_2$ such that each edge of $U_2$ is labeled by one of $\{u_1,u_2,u_3,u_4\}$ and the labels of two different edges incident on the same vertex are different. Now we replace each edge of $U_1$ joining a pair of endpoints  by a pair of oppositely oriented edge connecting the same pair endpoints. This pair of edges is labeled by the same letter, but  note that  their orientations are different. Let the resulting graph be $U'_1$. Edges of $U'_1$ come in pairs as described. Each such pair of edge form a \emph{standard bigon} in $U'_1$.
	We define $U'_2$ similarly. Then the action $L\acts U_1\times U_2$ gives a free and cocompact action $L\acts U'_1\times U'_2$ which preserves orientations of edges (though the action might not preserve edge labels). Moreover, the second assumption of the lemma implies that there exists a vertex $x'\in U'_2$ such that the $L$-stabilizer of $U'_1\times\{x'\}$ acts on $U'_1\times \{x'\}$ in a label-preserving way.
	
	The universal cover of $U'_1\times U'_2$ with the pulled back edge labels and orientations is exactly $T_h\times T_v$ (this follows from the fact that the valence of each vertex has doubled in going from $U_i$ to $U_i'$ for $i=1,2$).  Let $H$ be the group of automorphisms of $T_h\times T_v$ given by lifts of automorphisms of $U'_1\times U'_2$ coming from the $L$-action. Then there is an exact sequence:
	$$
	1\to \pi_1((U'_1\times U'_2)/L)\to H\to L\to 1.
	$$
	Note that $H$ acts properly and cocompactly on $T_h\times T_v$ preserving orientations of edges. As $L$ is irreducible, $H$ is irreducible. As $L$ sends standard bigons to standard bigons, $H$ sends standard lines to standard lines. The second requirement of Lemma~\ref{lem:tree} follows from the last sentence of the previous paragraph.
\end{proof}	

\begin{lemma}
	\label{lem:rf}
	Suppose there is a torsion free irreducible lattice $L_0$ acting on $U_1\times U_2$ cocompactly such that $H$ is residually finite. Then there exists $L$ satisfying the requirements of Lemma~\ref{lemma:smaller tree}. 
\end{lemma}

\begin{proof}
	After passing to an index two subgroup if necessary, we can assume that $L_0$ preserves  the two factors. Take an arbitrary vertex $z\in U_2$, let $L_{z}$ be the $L_0$-stabilizer of $U_1\times \{z\}$. Let $U_{1,z}$ be defined as before. Then there is a finite cover $\hhat U_{1,z}$ of $U_{1,z}$ such that 
	\begin{itemize}
		\item each edge in $\hhat U_{1,z}$ is embedded;
		\item it is possible to label each edge of $\hhat U_{1,z}$ using labels from $\{w_1,w_2,w_3,w_4\}$ such that the labels of two different edges incident on the same vertex of $\hhat U_{1,z}$ are different.
	\end{itemize}
	Indeed, the first condition can be arranged using the residual finiteness of $L_z$, which is a finitely generated free group. The second item can be arranged using Leighton's graph covering theorem (which says that two finite graphs with a common cover have a common finite cover, see \cite{MR693362}). Note that any further finite covers of $\hhat U_{1,z}$ will satisfy the above two items as well.
	
	Let $\hhat L_z$ be the finite index subgroup of $L_z$ corresponding to the covering $\hhat U_{1,z}\to U_{1,z}$. As $L_0$ is residually finite, we can find a finite index subgroup $\hhat L_0$ of $L_0$ such that $\hhat L_0\cap L_z\subset \hhat L_z$. Now consider the action $\hhat L_0\acts U_1\times U_2$, and let $U'_{1,z}$ be the quotient of $U_1\times \{z\}$ by the $\hhat L_0$-stabilizer of $U_1\times\{z\}$. Then $U'_{1,z}$ is a finite cover of $\hhat U_{1,z}$ (as defined in the previous paragraph), hence $U'_{1,z}$ satisfies both  conditions in the previous paragraph. Thus we are done by choosing $L=\hat L_0$.
\end{proof}

Now we are ready to prove Lemma~\ref{lem:tree}.
\begin{proof}[Proof of Lemma~\ref{lem:tree}]
	By the main theorem of \cite{borel-harder}, there exist irreducible (necessarily uniform) lattices in finite products of p-adic simple groups. 
	In particular, there exists a residually finite torsion free irreducible lattice acting cocompactly on $U_1\times U_2$. Now we 
	are done by Lemma~\ref{lem:rf} and Lemma~\ref{lemma:smaller tree}.
\end{proof}

\subsection{Extending actions over Salvetti complexes}\label{sec-hughes}
In this subsection and the next, we adapt a construction due to Hughes \cite[Section 7.2]{hughes2022graphs}.
Let $\VV X_\GG$ be the vertex set of $X_\GG$. Given a group $H$ acting by flat-preserving automorphisms on $X_\GG$, we have a \emph{type cocycle} $c: H\times \VV X_\GG\to \aut(\GG)$ with $h\in H$ and $x\in \VV X_\GG$ defined as follows. For a vertex $v\in \GG$, let $\ell$ be the standard line of type $v$ containing $x$. Then $c(h,x)(v)$ is defined to be the type of $h(\ell)$. One readily verifies that $c(h,x):V\GG\to V\GG$ preserves adjacency of vertices, hence extends to an automorphism of $\GG$; and $c(h,x)$ is indeed a cocycle.

Let $\GG_1$ be an induced subgraph of $\GG_2$. Then there is a locally isometric embedding $S_{\GG_1}\to S_{\GG_2}$. By pre-composing this with the covering map $X_{\GG_1}\to S_{\GG_1}$, we obtain $X_{\GG_1}\to S_{\GG_2}$. While $X_{\GG_1}\to S_{\GG_2}$ is not a covering map, we can ``complete'' it to a covering map as follows. Consider the homomorphism $\Gamma_{\GG_2}\to \Gamma_{\GG_1}$ fixing each generator in $\GG_1$ and sending all generators in $\GG_2\setminus\GG_1$ to identity. Let $K$ be the kernel of this homomorphism and $Z$ be the cover of $S_{\GG_2}$ corresponding to $K$. Then there is an embedding $X_{\GG_1}\to Z$. Under such an embedding, $X_{\GG_1}$ and $Z$ have the same vertex set. Moreover, we can obtain the 1-skeleton of $Z$ from the 1-skeleton of $X_{\GG_1}$ by attaching a collection of edge loops to each of the vertices of $X_{\GG_1}$, one edge loop for each vertex outside $\GG_2\setminus \GG_1$. The complex $Z$ is called the \emph{canonical completion} of $X_{\GG_1}$ with respect to the locally isometric embedding $X_{\GG_1}\to S_{\GG_2}$. This is a special case of more general construction by Haglund and Wise \cite{HaglundWise2008}.

\begin{lemma}
	\label{lem:extend2}
	Let $\GG_1,\GG_2,Z$ be as above.
	Suppose $H$ is a group acting freely and cocompactly on $X_{\GG_1}$ such that it preserves orientations of edges, and sends standard flats to standard flats.
	Let the associated type cocycle be $c_1:H\times \VV X_{\GG_1}\to \aut(\GG_1)$. 
	
	Suppose that there exists a cocycle $c_2:H\times \VV X_{\GG_1}\to \aut(\GG_2)$ such that $c_2(h,x)_{\mid \GG_1}=c_1(h,x)$ for any $h\in H$ and $x\in \VV X_{\GG_1}$. Then the action $H\acts X_{\GG_1}$ extends to an action $H\acts Z$ preserving edge orientations, such that $h\in H$ sends an edge loop based at $x\in X_{\GG_1}$ labeled by $v\in \GG_2\setminus \GG_1$ to an edge loop based at $h(x)$ labeled by $c_2(h,x)(v)$.
\end{lemma}

\begin{proof}
	The description in the lemma gives a well-defined action of $H$ on the 1-skeleton $Z^{(1)}$ of $Z$ that preserves edge orientations. The action $H\acts Z^{(1)}$ preserves $X^{(1)}_{\GG_1}$. Moreover, it send a pair of edges with commuting labels based at the same vertex to another pair of edges with commuting labels. Thus $H\acts Z^{(1)}$ extends to an action of $H$ on the 2-skeleton of $Z$. One readily checks that the action extends to higher skeleta as well.
\end{proof}

\subsection{Extending the action of the seed lattice}	\label{sec-extn}

%

Let $\GG$ be a finite simplicial graph, with a base vertex $v$ and a vertex $w\neq v$. Let $\LL$ be a wedge of four copies of $\GG$ along $v$. Then there is a folding map $f:\LL\to \GG$ sending each copy of $\GG$ to $\GG$ via the "identity" map. Let $\LL'=\{w_1,w_2,w_3,w_4\}=f^{-1}(w)$. Let $\GG_0,\LL_0,\LL'_0$ be the join of $\GG,\LL,\LL'$ with $\{u_1,u_2,u_3,u_4\}$.

There is an action of the permutation group on four elements $S_4$ on $\LL$, permuting the four copies of $\GG$'s in $\LL$. Hence $S_4$ also permutes the set $\{w_1,w_2,w_3,w_4\}$. In particular, this gives an injective homomorphism $\aut(\LL')\to \aut(\LL)$, which also induces an injective homomorphism $\aut(\LL'_0)\to\aut(\LL_0)$. The folding map $f:\LL\to \GG$ extends to a folding map $\LL_0\to \GG_0$, which we still denote by $f$.

Let $H\acts T_h\times T_v=X_{\LL'_0}$ be as in Lemma~\ref{lem:tree}. Consider the locally isometric embedding $X_{\LL'_0}\to S_{\LL_0}$ which is a composition of the covering map $X_{\LL'_0}\to S_{\LL'_0}$ and the embedding $S_{\LL'_0}\to S_{\LL_0}$. Similarly, we define $X_{\LL'}\to S_\LL$.
Let $Z_{\LL_0}$ (resp. $Z_{\LL'}$) be the canonical completion of $X_{\LL'_0}$ (resp. $X_{\LL'}$) with respect to $X_{\LL'_0}\to S_{\LL'_0}\to S_{\LL_0}$ (resp. $X_{\LL'}\to S_{\LL'}\to S_\LL$). Then $Z_{\LL_0}=Z_{\LL}\times T_v$.

Let $c_{\LL'_0}:H\times \VV X_{\LL'_0}\to \aut(\LL'_0)$ be the type cocycle associated to the action $H\acts T_h\times T_v$. Using the monomorphism $\aut(\LL'_0)\to\aut(\LL_0)$ as above, we can extend the cocycle $c_{\LL'_0}$ to a new cocycle $c_{\LL_0}:H\times \VV X_{\LL'_0}\to \aut(\LL_0)$. As in Lemma~\ref{lem:extend2}, the cocyle $c_{\LL_0}$ gives a free and cocompact action $H\acts Z_{\LL_0}=Z_{\LL}\times T_v$ preserving edge orientations. Lemma~\ref{lem:tree} (2) and the construction in Lemma~\ref{lem:extend2} imply that there exists a vertex $x\in T_v$ such that the action of the stabilizer of $Z_\LL\times T_v$ on  $Z_\LL\times T_v$ is both label-preserving and orientation-preserving.

Given an edge $e$ of $Z_{\LL_0}$, we will also call its label the \emph{$\LL_0$-label}. Its \emph{$\GG_0$-label} is defined to be the image of the label of $e$ under the map $f:\LL_0\to \GG_0$. For any vertex $v\in \LL_0$ and any automorphism $\alpha\in \aut(\LL_0)$ in the image of $\aut(\LL'_0)\to \aut(\LL_0)$, we know that $f(v)=f(\alpha(v))$. Thus the action $H\acts Z_{\LL_0}$ preserves the $\GG_0$-label of edges in $Z_{\LL_0}$ (though their $\LL_0$-label might not be preserved under the $H$-action).

Consider the factor action $H\acts Z_\LL$ of $H\acts Z_\LL\times T_v$. We define the $\LL$-labels and $\GG$-labels of edges in $Z_\LL$ in a similar way as before. Then the factor action $H\acts Z_\LL$ preserves orientations of edges and $\GG$-labels of edges.

\subsection{Bestvina-Brady groups}\label{sec-bb} We use Bestvina-Brady Morse theory \cite{bb-morse} in this subsection to construct normal subgroups with higher finiteness properties.
By the discussion in the previous subsection, each edge of $Z_{\LL}/H$ has a well-defined orientation and $\GG$-label. Now we define a cellular map $\pi:Z_\LL/H\to S_\GG$ which is orientation-preserving and $\GG$-label preserving as follows:  $\pi$ maps each vertex of $Z_\LL/H$ to the base vertex of $S_\GG$, and  each edge in $Z_\LL/H$ to the loop in $S_\GG$ with the same label in an orientation-preserving way. Note that the folding map $\LL\to \GG$ sends complete subgraphs in $\LL$ bijectively to complete subgraphs in $\GG$. It follows from the fact that the action $H\acts Z_\LL$ preserves orientations and $\GG$-labels of edges that the map $\pi$ extends to higher skeleta.

Let $W=Z_{\LL_0}/H=(Z_\LL\times T_v)/H$ and $G=\pi_1(W)$.
Let 
$\widetilde Z_\LL$ be the universal cover of $Z_\LL$. Then $\widetilde Z_{\LL_0}=\widetilde Z_\LL\times T_v$.
Note that $\widetilde Z_{\LL_0}/G=W$. 

The group $G$ can be alternatively described as follows. 
It is  the collection of all automorphisms of $\widetilde Z_{\LL_0}$ which are lifts of automorphisms of $Z_{\LL_0}$ coming from the action of $H$.  In particular, we have the following short exact sequence:
$$
1\to \pi_1(Z_{\LL_0})\to G\to H\to 1.
$$

By the discussion in the previous subsection, there exists a vertex $x\in T_v$ such that the action $\stab_H(Z_\LL\times \{x\})\acts Z_\LL\times\{x\}$ preserves orientations and labels of edges. Let $W_\LL$ be the quotient space under this action and let $J=\pi_1 W_\LL$. Then each edge of $W_\LL$ has a well-defined orientation and labeling by vertices of $\LL$.
There is a locally isometric embedding $W_\LL\to W$,  inducing an embedding of groups $J\to G$. Thus we will view $J$ as a subgroup of $G$.

Take a copy of $\mathbb S^1$, viewed as a graph with a single vertex and a single edge. We choose an orientation of the edge in $\mathbb S^1$. Let $S_\GG\to \mathbb S^1$ be the map sending each edge of $S_\GG$ to the edge in $\mathbb S^1$ in an orientation-preserving way, and then extending linearly to higher skeleton. This map is called the \emph{Bestvina-Brady map}\cite{bb-morse}. It induces a group homomorphism $\Gamma_\GG\to \mathbb Z$ sending each generator of $\Gamma_\GG$ to $1\in \mathbb Z$. The kernel of this homomorphism is the \emph{Bestvina-Brady group} of $S_\GG$\cite{bb-morse}. 

Let $\Theta: W=(Z_\LL\times T_v)/H\to Z_\LL/H\to S_\GG\to \mathbb S^1$ be the composition of the map induced by the projection $Z_\LL\times T_v\to Z_\LL$, 
the map $\pi$, and the Bestvina-Brady map.
Let $\widehat G$ be the kernel of the homomorphism $G\to \mathbb Z$ induced by $\Theta$, and let $\widehat J=\widehat G\cap J$.

\begin{lemma}
	\label{lem:fi}
	$\widehat J$ is finite index in the Bestvina-Brady group of $\Gamma_\LL$.
\end{lemma}

\begin{proof}
	Note that the composition $\theta:W_\LL\to W\stackrel{\Theta}{\to} \mathbb S^1$ sends each oriented edge to an oriented edge in an orientation-preserving way (since  this is true for each step in the definition of $\theta$). On the other hand, as the universal cover of $W_\LL$ is a copy of $X_\LL$, with $J$ acting on $X_\LL$ freely preserving edge labels and orientations, we know that $J$ is a finite index subgroup of the group of deck transformations of the cover $X_\LL\to S_\LL$. In particular, this gives a label and orientation preserving covering map $g_1:W_\LL\to S_\LL$. Let $g_2:S_\LL\to \mathbb S^1$ be the Bestvina-Brady map. Then $g_2\circ g_1$ preserves orientations of edges. Thus $g_2\circ g_1$ and $\theta$ are the same when restricted to the 1-skeleton of $W_\LL$. Hence $\theta=g_2\circ g_1$, and the lemma follows.
\end{proof}

\subsection{Indiscrete commensurator and finiteness properties}\label{sec-fin}
We consider the action $G\acts \widetilde Z_{\LL_0}=\widetilde Z_\LL\times T_v$. The induced action $G\acts \widetilde Z_\LL$ on the first factor induces a homomorphism $\Phi: G\to \aut(\widetilde Z_\LL)$. Note that $\Phi(\widehat J)$ is a subgroup of $\Phi(J)$ which is infinite index, infinite and normal.

\begin{theorem}\label{thm-bb}
	The following are true.
	\begin{enumerate}
		\item $\Phi(G)$ is not discrete in $\aut(\widetilde Z_\LL)$.
		\item  $\Phi(J)$ is a cocompact lattice in $\aut(\widetilde Z_\LL)$.
		\item $\Phi(\widehat J)$ is an infinite index, and infinite normal subgroup of $\Phi(J)$. The common commensurates of $\Phi(\widehat J)$ and $\Phi(J)$ in $\aut(\widetilde Z_\LL)$ is indiscrete. 
		\item $\Phi(\widehat J)$ is finite index in the Bestvina-Brady group of $\Gamma_\LL$.
		\item $\Phi(\widehat J)$ does not contain a finite index subgroup normalized by a finite index subgroup of the commensurator of $\Phi(\widehat J)$ in $\aut(\widetilde Z_\LL)$. 
	\end{enumerate}	
\end{theorem}

\begin{proof}
	For Assertion (1), as $H\acts T_h\times T_v$ is irreducible, we know that the projection of $H$ to $\aut(T_h)$ is not discrete. So when we pass to the action $H\acts Z_{\LL}\times T_v$, the projection of $H$ to $\aut(Z_\LL)$ is not discrete. Now (1) follows. Assertion (2) is clear. 
	For Assertions (3) and (5) follow from Proposition~\ref{prop-startingpoint}. Assertion (4) follows from Lemma~\ref{lem:fi}.
\end{proof}

The following is an immediate corollary.
\begin{cor}\label{cor-bb}
	Let $\GG$ be a finite simplicial graph, with a base vertex $v$ and a vertex $w\neq v$. Let $\LL$ be a wedge of four copies of $\GG$ along $v$. Let $X_\LL$ be the universal cover of the Salvetti complex of $\LL$. Let $BB_\LL$ and $\Gamma_\LL$ be the Bestvina-Brady subgroup and the right-angled Artin group associated with $\LL$, viewed as subgroups of $G=\aut(X_\LL)$. Then
	$\comme_G(BB_\LL)\cap \comme_G(\Gamma_\LL)$ is not discrete. In particular, $\comme_G(BB_\LL)$ is not discrete.
\end{cor}

\section{Common Commensurators}\label{sec-cc}

\subsection{Preserving lines with holes and Common Commensurators}\label{sec-holes}

Let $X$ be a locally finite contractible cell-complex, and $G = \aut(X)$. (For the purposes of this paper, $X$ will mostly be a tree or the universal cover of a Salvetti complex.)
Let $\Gamma < G$, let $K<\Gamma$ be a normal subgroup, and let $g\in \comme_G(\Ga)$.

We write $Q=\Gamma/K$ for the quotient group.
Conjugating by $g\in G$, we obtain groups $K^g<\Gamma^g$ and a corresponding quotient $Q^g:=\Gamma^g/K^g$.

\begin{defn}\label{def-gammawholes}
For $\gamma \in \Gamma$, we shall refer to the cyclic group $\langle \gamma \rangle$ as a \emph{ $\gamma-$line in $\Gamma$}. Further, any finite index subgroup $\langle \gamma^N \rangle$ of $\langle\gamma\rangle$
will be referred to as a 
\emph{ $\gamma-$line with holes.}  A  \emph{$\Gamma-$line with holes} is a 
$\gamma-$line with holes for some $\gamma \in \Gamma$.
\end{defn}
For all applications, $\gamma$ will be of infinite order.
For any $\gamma \in \Gamma$ and $g\in \comme_G(\Ga)$,
there exists a positive integer $N$, such that $(\gamma^g)^N \in \Gamma$. Hence, for any $\gamma \in \Gamma$, and
$ g\in \comme_G(\Ga)$, the conjugation action by $g$ sends   $<\gamma^N>$ for some
$N \in \natls$ to a  $\Gamma-$line with holes.

\begin{defn}\label{def-pa}\cite{koberda-mj}
	We say an element $g\in \comme_G(\Ga)$ \emph{preserves $Q$--lines with holes} if for all $\gamma \in \Gamma$
	there exists an integer $N>0$ such that
	\[\gamma^n \equiv (\gamma^n)^g  \pmod  K\] for all $n\in N\Z$. That is, there exists $N > 0$ such that $x_m=[\gamma^{mN},g] \in K$
	for all $m\in\Z$.
\end{defn}

\begin{theorem}\label{thm:comm-nonab}\cite[Theorem 3.4]{koberda-mj}
	Let $\Ga < G$, let $K$ be a normal subgroup of $\Ga$, and let $Q = \Ga/K$.
	Suppose that \[g \in \comme_G \Ga \cap \comme_G K.\] Then
	$K^g$ preserves $Q$--lines with holes.
\end{theorem}

For the rest of this subsection, let $G$ be a locally compact second countable topological group. Let $\Ga < G$ be a lattice, and let $K <\Gamma$ be a subgroup.

\begin{defn}\label{def-commoncomm} An element
$g \in G$ is said to be a \emph{common commensurator} of $\Gamma, K$, if
$g \in \comme_G (K) \cap \comme_G (\Gamma)$.
\end{defn}

We recall a Theorem due to Borel that any commensurator in a semi-simple Lie group of a 
Zariski dense subgroup $K$ of an arithmetic lattice $\Gamma$ is a common commensurator.

\begin{prop}\label{commens=commesn}\cite[Theorem 2]{borel}  \cite[p. 123]{zimmer-book}
	Let $\Ga < G$ be an arithmetic lattice in a  semi-simple algebraic $\Q$--group and let $K < \Ga$ be
	a Zariski dense subgroup. Then $\comme_G(K)<\comme_G(\Ga)$.
	Suppose furthermore that the center of $G$ is trivial. Then
	$\comme_G(\Gamma)$ coincides with the $\Q$--points of $G$.
\end{prop}

The following theorem of Fisher-Mj-van Limbeek ensures that
any commensurator  in a semi-simple Lie group of a 
Zariski dense \emph{normal} 
subgroup $K$ of an arithmetic lattice $\Gamma$ is a common commensurator.

\begin{prop}\label{fmv-normal}\cite{fmv}
Let $K<\Gamma$ be a Zariski dense normal subgroup of a lattice in a rank one Lie group $G$. If $K$ has indiscrete commensurator, $\Gamma$ is arithmetic. Hence, $\comme_G(K)<\comme_G(\Ga)$.
\end{prop}

Proposition~\ref{fmv-normal} was used to establish the following main theorem 
of \cite{fmv}:
\begin{theorem}
\label{fmv-normalthm}
Let $K<\Gamma$ be a thin normal subgroup of a lattice in a rank one Lie group $G$. Then $\comme_G(K)$ is discrete.
\end{theorem}

\subsection{Burger-Mozes examples}\label{sec-bm} The purpose of this subsection is to investigate an example due to Burger and Mozes \cite{bm-divergence}. Let $T_{2d}$ denote the regular $2d-$valent tree, and $G=Aut (T_{2d})$. 
Let $K_d < G$ be the fundamental group of the Cayley graph of $\Z^d$ with respect to the standard generators thought of as deck transformations of $T_{2d}$ with quotient the Cayley graph of $\Z^d$ w.r.t. the standard generators. 

\begin{prop}\cite[Proposition 8.1]{bm-divergence}\label{prop-bm}
The commensurator of $K_d $ is dense in $G$.
\end{prop} 

\begin{rmk}\label{rmk-bm}
We note, in passing, that \cite[Theorem 8.3]{bm-divergence} deals with a tree-lattice
and proves that it has dense commensurator  in $G$. In fact, that the  example in \cite[Theorem 8.3]{bm-divergence} is a non-uniform tree lattice  can be seen by computing the volume, which is given by $\frac{1}{b} + \sum_1^\infty \frac{1}{2^i}$
(where we are borrowing the notation from \cite[Theorem 8.3]{bm-divergence}).

Our focus here, however, is on infinite covolume discrete subgroups of $G$, and hence we analyze only the example from
Proposition~\ref{prop-bm}. 
\end{rmk}

The Burger-Mozes examples $K_d$ in Proposition~\ref{prop-bm} are normal subgroups 
 of lattices $\Gamma$ in $G=\aut(T_{2d})$, where $\Gamma$ is a free group on $d$ generators. Despite the fact that $K_d$ has dense commensurator in $G$, 
  the \emph{common commensurator} of $K_d, \Gamma$ turns out to be discrete as we show below. We shall need an analog of \cite[Lemma 2.6]{koberda-mj} before proceeding;
  the proof is quite different from \cite{koberda-mj} and uses work of Bass-Lubotzky
  \cite[Proposition 6]{bl-book}.
  Let  $G, \Gamma, K_d$ be as above. Let  $L=\comme_G(K_d) \cap \comme_G(\Gamma)$.
  Let $L_0 = \langle K_d^l \rangle$ denote the subgroup of $L$ generated by the subgroups $K_d^l$  
  as $l$ ranges over $ L$.   Clearly, $L_0 < L$. Let $L_1 =   \langle L_0, \Gamma \rangle$
  be the subgroup of $L$ generated by $L_0$ and $\Gamma$ (clearly, $\Gamma < \comme_G(K_d) \cap \comme_G(\Gamma)=L$). 
  
  \begin{lemma}\label{lem-suffdiscrete}
  Let $L, L_1$ be as above. Then  $L$ is discrete if and only if $L_1$ is discrete.
  \end{lemma}

  \begin{proof} Since $L_1 < L$, the only if direction is immediate.
  	
  Conversely, suppose that $L_1$ is discrete. Since
  $\Gamma < L_1$, and $\Gamma$ is a lattice in $G$ acting transitively on $T_{2d}$, 
  it follows that $[L_1: \Gamma] < \infty$. Hence there exist finitely many coset representatives $l_1, \cdots l_k$ of $\Gamma$ in $L_1$, and hence finitely many
  conjugates $K_d^{l_i}, i=1, \cdots, k$ of $K_d$ in $L_1$.

  If $L$ is indiscrete, there exists an infinite sequence $g_i \in L$ such that $g_i$
  converges to the identity. Then, passing to a subsequence if necessary, we may assume that $K_d^{g_i}=K_d^{g_j}$ for all $i, j$ since $L_1$ is discrete. Finally, replacing $g_i$ by
  $g_i^{-1}g_{i+1}$, we may assume that $K_d^{g_i}=K_d$ for all $i$. It follows that 
  $K_d$ has indiscrete normalizer $N_G(K_d)$. By \cite[Proposition 6.2(b)]{bl-book},
  $N_G(K_d)$ is a closed, and hence closed indiscrete subgroup of $G$. 
  
 We now show that $N_G(K_d)$ is actually discrete, which gives a contradiction. Let $\GG_d$ be the Cayley graph of $\mathbb Z^d$. Then $\GG_d$ is isomorphic to the orbit space $T_{2d}/K_d$. Let $h:T\to T$ be an element in $N_G(K_d)$. As $h$ sends $K_d$ orbits to $K_d$ orbits. we know $h$ descends to an isometry $\bar h:\GG_d\to \GG_d$. Thus $N_G(K_d)$ fits into an exact sequence $1\to K_d\to N_G(K_d)\to G_d\to 1$, where $G_d$ is the automorphism group of $\GG_d$. As $G_d$ is discrete, we know $N_G(K_d)$ is discrete.
  
  \end{proof}

  \begin{prop}\label{prop-disccomm}
Let $G, \Gamma, K_d$ be as above. Then $L=\comme_G(K_d) \cap \comme_G(\Gamma)$ is discrete.
  \end{prop}
  
  \begin{proof}
   Let $Q=\Gamma/K_d$ denote the quotient group, so that $Q$ is isomorphic to  $\Z^d$. By Theorem~\ref{thm:comm-nonab}, 
  $K_d^l$  preserves $Q-$lines with holes
 for all $l \in L$. Let $L_0 = \langle K_d^l \rangle$ denote the subgroup of $L$ generated by the subgroups $K_d^l$  
as $l$ ranges over $ L$. As noted in \cite[Observation 3.3]{koberda-mj}, $L_0$ 
 preserves $Q-$lines with holes. Since $Q$ is abelian, the conjugation action
 of $\Gamma$ on itself descends to a trivial action on $Q$. Hence, $L_1 =\langle L_0, \Gamma\rangle$  preserves $Q-$lines with holes by \cite[Observation 3.3]{koberda-mj} again. 
 
By Lemma ~\ref{lem-suffdiscrete}, the proposition would follow if we could prove that $L_1$ is discrete. Suppose not. Then there exists non-trivial $g \in L_1$ arbitrarily close to the identity of $G$. In particular, $g$ fixes some vertex of $T_{2d}$. We identify vertices of $T_{2d}$ with elements in $\Gamma$, and suppose $g$ fix the vertex which corresponds to $\beta\in \Gamma$.

Let $a_1, \cdots, a_d$ denote the standard generating set for $\Gamma$. As $T_{2d}$ is the Cayley graph of $\Gamma$, edges of $T_{2d}$ are oriented and labeled by generators of $\Gamma$.
Let $a'_i=\beta a_i\beta^{-1}$.
Then there exists $N \in \natls$ such that $$(a'_i)^{mN}
\equiv ((a'_i)^{g})^{mN}
\mod K_d$$ for all $i = 1, \cdots, d$ and $m \in \Z$. 

Let $\ell_i$ be the axis of $a'_i$ acting on $T_{2d}$. Then $\beta\in \ell_i$ and $\ell_i$ is also the axis of $(a'_i)^{N}$.
Hence $g\ell_i$ is the axis of $((a'_i)^{g})^{mN}$ 
Consider the geodesic edge path $\sigma_i$ in $\ell_i$ from $\beta$ to $\beta a_i^N$ in $T_{2d}$. Then $((a'_i)^{g})^{N}=\beta b_i\beta^{-1}$, where $b_i$ is the word in $\Gamma$ read from the edge path $g\sigma_i$ using orientation and labeling of edges. As $b_i$ and $a_i^N$ have the same image under $\Gamma\to \mathbb Z^d$ and $b_i$ is a word of length $N$, we must have $b_i=a^N_i$, which implies $g\sigma_i=\sigma_i$ and $g\ell_i=\ell_i$ for all $i$. Thus $g$ fixes the closed 1-ball around $\beta\in T_{2d}$ pointwise.




 
 
 
By induction on distance from $\beta\in T_{2d}$ and repeating the previous argument, we know
$g$ is the \emph{identity} map on  $T_{2d}$ and so $K_d^l$ must intersect some compact open subgroup of $G$ trivially. Equivalently, 
  $\comme_G(K_d) \cap \comme_G(\Gamma)$ is discrete.
  \end{proof}

 \begin{rmk}\label{rmk-bm2}
 Proposition~\ref{prop-disccomm} shows that the common commensurator of $\Gamma$ and $K_d$ is discrete, while Proposition~\ref{prop-bm} shows that the  commensurator of  $K_d$ alone in $G$ is dense. The argument in \cite{bm-divergence} proceeds by passing to a finite-sheeted cover of the Cayley graph $\QQ$ and then switching
 labels on adjacent edges in a finite ball. These elements of $G$
 \emph{ do not} come from a commensuration of the 
 free group $\Gamma$, as they do not send a finite index subgroup to another finite index
 subgroup. Nevertheless, both $\comme_G(\Gamma)$ and $\comme_G(K_d)$ are individually
 dense in $G$. 
 Thus, what Proposition~\ref{prop-disccomm} in conjunction with Proposition~\ref{prop-bm} shows is that Borel's statement Proposition~\ref{commens=commesn} fails for $G=\aut(T_{2d})$, underscoring the difference between $G$ and a semi-simple or p-adic Lie group. The hypothesis of an ambient semi-simple or p-adic Lie group in Theorem~\ref{fmv-normalthm} is thus important. 
 \end{rmk}

\subsection{Questions and Remarks}\label{sec-dhf}
We conclude with some questions and remarks aimed at extending the Greenberg-Shalom hypothesis \cite{bfmv} to arbitrary isometry groups of CAT(0) spaces and Gromov-hyperbolic spaces. We refer to \cite{caprace-monod1,caprace-monod2} for details on the structure theory of isometry groups of such spaces. For the purposes of this section $W$ will denote a complete geodesic CAT(0) space or a 
Gromov-hyperbolic space. Also, $G = Isom(W)$ will denote the group of isometries of $W$. In case $W$ is simplicial or cubical, we demand that the isometries preserve the  
simplicial or cubical structure.

\begin{defn}  \cite{caprace-monod1,caprace-monod2}
A subgroup $K < G$ is said to be \emph{geometrically dense} if 
\begin{enumerate}
\item $K$ does not leave any convex subspace of $W$ invariant when $W$ is CAT(0).
\item  the limit set of $K$ is all of $\partial W$ when $W$ is Gromov-hyperbolic.
\item $K$ does not have a global fixed point on $\partial W$.
\end{enumerate}
\end{defn}

Theorem~\ref{thm-wisecsc}  provides an example of a normal subgroup  $K<\Gamma < G$, where the (closure of the) common commensurator is indiscrete but acts cocompactly on the  tree $T_h$. Theorem~\ref{thm-bb} refines this example by giving a more sophisticated example satisfying strong finiteness properties. However, in both these cases, since $K<\Gamma$ is normal, the common commensurator contains $\Gamma$ and is therefore cocompact in $G$. We do not know if this is a general phenomenon.
\begin{qn}\label{qn-harmonic}
Let $K<\Gamma < G$ be a discrete, geometrically dense subgroup of
a lattice $\Gamma$ in $G$ such that the
closure $\mathsf{G}$ of the
 common commensurator of $K, \, \Gamma$ is indiscrete in $G$. Is $\mathsf{G}$ 
 necessarily cocompact in $G$? 
\end{qn}
Of course, if $K<\Gamma$ is itself a commensurated subgroup, then $\Gamma$ lies in the common commensurator and hence  $\mathsf{G}$  is necessarily cocompact. Thus, Question~\ref{qn-harmonic} really
asks for examples where $K$ is not a commensurated subgroup of the ambient lattice
$\Gamma$.

The next question asks about the existence of irreducible subgroups of a product of tdlc groups. We refer the reader to \cite{bfmv} for the basics of Schlichting completions and Cayley-Abels graphs.
Let $\Gamma$ be a finitely generated group commensurating $K_1, K_2 < \Gamma$ with
\begin{enumerate}
\item $<K_1, K_2> = \Gamma$
\item $K_1 \cap K_2$ trivial.
\end{enumerate}
Let $\Gamma\big/\!\big/K_1, \Gamma\big/\!\big/K_2$ denote the Schlichting completions.
Let $\GG_1, \GG_2$ be the Cayley-Abels graphs of the tdlc groups $\Gamma\big/\!\big/K_1, \Gamma\big/\!\big/K_2$, respectively. Then $\Gamma$ acts properly on   $\GG_1 \times \GG_2$, since the stabilizer of $(v,w) \in \GG_1 \times \GG_2$ is the intersection $Stab(v) \cap Stab (w)$. Hence, it is an intersection of $K_1$ with a conjugate of $K_2$ and the latter is commensurable with $K_2$. Hence $Stab(v) \cap Stab (w)$ contains a finite-index subgroup that is trivial. Since $\Gamma$ is assumed torsion-free, the stabilizer is trivial. Hence the $\Gamma$ action is free, proper on 
 $\GG_1 \times \GG_2$.

\begin{qn}
Do there exist examples as above that are not lattices in $ \GG_1 \times \GG_2$? i.e.\ does there exist $\Gamma$ as above, commensurating two subgroups $H_1, H_2$ and satisfying the two above conditions, such that under the factor-wise inclusion maps into
$ \GG_1 \times \GG_2$, $\Gamma$ is not a lattice in $ \GG_1 \times \GG_2$?
\end{qn}

We now quickly translate the approach of \cite{koberda-mj} to the discrete/simplicial setup of the present paper. In \cite{eckmann-harmonic}, harmonic  forms and the Hodge theorem were translated to the simplicial setup.
We now use this to interpret some of the results of this paper.
In the proof of Proposition~\ref{prop-startingpoint}, $K$ 
and $\pi_1(M)$ are commensurated by $F_v$.
By Theorem~\ref{thm:comm-nonab}, if $g$ commensurates both 
$K$ 
and $\Ga$, $K^g$ preserves $Q$--lines with holes. Here, $Q$ is infinite cyclic.
In particular, this may be interpreted as saying that $K^g$   preserves the rational cohomology class given by $[\phi]$.
Let $\omega$ denote the unique (discrete) harmonic form representing $[\phi]$.
(This is where we use the simplicial Hodge theorem of \cite{eckmann-harmonic}.)
Then \cite{koberda-mj},  $F_v$ preserves $\omega$ \emph{on the nose}. Hence, so does
$\bbar F_v$. Thus, the stabilizers of vertices of $T_h = \til{M}$ preserve 
$\omega$. The subgroup $\Omega$ of $\aut(T_h)$ generated by the collection $\{\bbar F_v: v \, {\rm a \, vertex\, of\, } T_h\}$ preserves $\omega$. Since each 
$\bbar F_v$ is indiscrete, so is $\Omega$.

\begin{qn}
	Do there exist examples $K<\Gamma < G=\aut{(W)}$ with $K$ normal in $\Gamma$ so that 
	\begin{enumerate}
	\item the closure $H$ of the common commensurator of $K, \Gamma$ is indiscrete, and
	\item $H$ does not preserve any discrete harmonic form on $W$.
	\end{enumerate}
\end{qn}

\end{document}